\documentclass[12pt]{article}
\usepackage{amsthm}
\usepackage{amsmath}
\usepackage{amsfonts}
\usepackage{graphicx}
\usepackage{latexsym}
\usepackage{amssymb}
\usepackage{lmodern}

\newcommand{\field}[1]{\mathbb{#1}}

\newcommand{\Z}{\field{Z}}

\newcommand{\cA}{{\cal A}}
\newcommand{\cB}{{\cal B}}
\newcommand{\cC}{{\cal C}}
\newcommand{\cD}{{\cal D}}

\newcommand{\cT}{{\cal T}}

\newcommand{\cR}{{\cal R}}

\newcommand{\cM}{{\cal M}}

\newcommand{\cH}{{\cal H}}

\newtheorem{theorem}{Theorem}
\newtheorem{lemma}{Lemma}

\newtheorem{example}{Example}

\begin{document}

\bibliographystyle{plain}

\title{
\begin{center}
An improved Recursive Construction\\
for Disjoint Steiner Quadruple Systems
\end{center}
}
\author{
{\sc Tuvi Etzion}\thanks{Department of Computer Science, Technion,
Haifa 3200003, Israel, e-mail: {\tt etzion@cs.technion.ac.il}.} \and
{\sc Junling Zhou}\thanks{Department of Mathematics,
Beijing Jiaotong University, Beijing, China,
e-mail: {\tt jlzhou@bjtu.edu.cn}.}}

\maketitle

\begin{abstract}
Let $D(n)$ be the number of pairwise disjoint Steiner quadruple systems.
A simple counting argument shows that $D(n) \leq n-3$ and a set of $n-3$ such
systems is called \emph{a large set}. No nontrivial large set was constructed yet, although it is known
that they exist if $n \equiv 2$ or $4~(\text{mod}~6)$ is large enough.
When $n \geq 7$ and $n \equiv 1$ or $5~(\text{mod}~6)$, we present a recursive construction and prove a recursive formula on $D(4n)$, as follows:
$$
D(4n) \geq 2n + \min \{D(2n) ,2n-7\}.
$$
The related construction has a few advantages over some of the previously known constructions
for pairwise disjoint Steiner quadruple systems.
\end{abstract}

\vspace{0.5cm}


\footnotetext[1] {This research was supported in part by the
111 Project of China (B16002) and in part by the NSFC grants 11571034 and 11971053.
Part of the research was performed during a visit of T. Etzion to Beijing Jiaotong University.
He expresses sincere thanks to the 111 Project of China (B16002) for its support and to the Department
of Mathematics at Beijing Jiaotong University for their kind hospitality.}

\newpage
\section{Introduction}
\label{sec:introduction}

A \emph{Steiner system of order $n$}, $S(t,k,n)$, is a pair $(Q,B)$, where $Q$ is an $n$-set
(whose elements are called \emph{points}) and $B$ is a collection
of $k$-subsets (called \emph{blocks}) of $Q$, such that each $t$-subset of $Q$ is contained
in exactly one block of $B$.

A \emph{large set} of Steiner systems $S(t,k,n)$, on an $n$-set $Q$, is a partition of all $k$-subsets
of $Q$ into Steiner systems $S(t,k,n)$. If we restrict ourself to Steiner systems $S(t-1,t,n)$, then exactly
two families of large sets are solved completely. A Steiner system $S(1,2,n)$ exists if and only if $n$ is even and its
large set is known as a one-factorization of the complete graph $K_n$. The existence of such one-factorizations
is a folklore and a survey can be found in~\cite{Wal97}. A Steiner system $S(2,3,n)$ is known as Steiner triple
system and the corresponding large set is known to exist for every admissible $n \equiv 1$ or $3 ~(\text{mod}~6)$, where $n > 7$.
It was first proved by Lu~\cite{Lu83,Lu84}, who left six open cases which were solved by Teirlinck~\cite{Tei91}.
An alternative shorter proof was given later by Ji~\cite{Ji05}.

A Steiner system $S(3,4,n)$ is also called a Steiner quadruple system of order~$n$
and it is denoted by SQS($n$). These systems which are
of special interest exist if and only if $n \equiv 2$ or $4~(\text{mod}~6)$~\cite{Han60}.
No construction of nontrivial large set of SQS($n$) is known, although it is a common belief that they exist
even for small $n$. Using probabilistic arguments, their existence for large enough $n$, $n \equiv 2$ or $4~(\text{mod}~6)$, was proved recently~\cite{Kee18}.
In practice, it is of interest to construct many disjoint systems.
Let $D(n)$ be the number of pairwise disjoint SQS($n$)s (PDQs of order $n$ in short, the
order is omitted when it is understood).
The size of an SQS($n$) is $\binom{n}{3} \Huge / 4$, and hence $D(n) \leq n-3$
and a large set contains $n-3$ PDQs of order $n$, i.e. $D(n)=n-3$, if a large set exists.

The first lower bound on $D(n)$ was given by Lindner~\cite{Lin77} who proved that
if $n \equiv 2$ or $4~(\text{mod}~6)$, then $D(2n) \geq n$. He further obtained
another bound~\cite{Lin85}, $D(4n) \geq 3n$ if $n \equiv 2$ or $4~(\text{mod}~6)$, where $n \geq 8$.
In a sequence of papers~\cite{KrMe74,Phe91,PhRo80} it was proved that $D(2n) \geq n$ for
$n= 5^a 13^b 17^c$, where $a,b,c \geq 0$ and $a+b+c>0$ and for $n=11$.
Etzion and Hartman~\cite{EtHa91} has proved that $D(2^k n) \geq (2^k -1)n$, $k \geq 2$, if
there exists a set of $3n$ PDQs of order $4n$ with a certain structure. The result can be applied
on $n=11$, $n \equiv 2$ or $4~(\text{mod}~6)$, and $n= 5^a 13^b 17^c$, where $a,b,c \geq 0$ and $a+b+c>0$.
In particular it is proved in~\cite{EtHa91} that $D( 5 \cdot 2^m ) \geq 5 \cdot 2^m -5$, which falls short
by just one system from a large set. This claim follows from the fact that $5 \cdot 2^m -4$ PDQs imply that the remaining
quadruples which are not contained in these PDQs form another SQS($5 \cdot 2^m )$.
The construction in~\cite{EtHa91} is recursive, based
on the existence of 2-chromatic PDQs. A recursive equation for $D(4n)$,
$n \equiv 1$ or $5~(\text{mod}~6)$, was given in~\cite{Etz93}, where it is proved that
$$
D(4n) \geq 2n + \min \{ D(2n),n-2 \}~.
$$
Our goal in this paper is to improve this recursive formula to the following theorem which
is the main result of the paper.

\begin{theorem}
\label{thm:main}
$~$
If $n \geq 7$ and $n \equiv 1$ or $5~(\text{mod}~6)$, then
$D(4n) \geq 2n + \min \{ D(2n),2n-7 \}$.
\end{theorem}
The rest of this paper is organized as follows. In Section~\ref{sec:review} we present the
two known recursive doubling construction for SQSs and for a set of PDQs.
The proof of Theorem~\ref{thm:main} is given in Section~\ref{sec:recursion} with its related construction.
We start with any SQS($2n$), a certain Latin square of order $2n$, and apply the first doubling
construction (DLS Construction) to obtain $2n$ PDQs of order $4n$. We analyse the quadruples which are contained in the
constructed $2n$ PDQs and continue with the second doubling construction (DB Construction) to form PDQs of order $4n$ which contain
quadruples which are not contained in the first $2n$ PDQs.
The construction for $n \equiv 5~(\text{mod}~6)$ is slightly simpler to explain than the one
for $n \equiv 1~(\text{mod}~6)$ and also
when $n$ is a prime some of the construction is slightly simpler. We make some distinction between
these cases in the proofs.
Examples of each step are presented throughout the description of the
related recursive construction. In Section~\ref{sec:conclusion} a conclusion, comparison with previous results,
and several problems for future research, are presented.
\vspace{-0.2cm}

\section{Two Recursive Constructions}
\label{sec:review}

In this section we will discuss two main recursive doubling constructions for
SQSs or more precisely, for a set of PDQs.
The first one is a doubling construction due to Lindner~\cite{Lin77}. It will
be called DLS (for Doubling Lindner Systems) Construction. We will present a slightly different variant from the one
given in~\cite{Lin77}. This variant was already presented in~\cite{Etz96}.
The second construction is a folklore doubling construction
which will be called DB (for doubling) Construction.
These two constructions will be the main ingredients
for the other constructions presented in this paper.

For the two constructions we need to define several well-known combinatorial designs
which are essential in the construction used to prove Theorem~\ref{thm:main} in Section~\ref{sec:recursion}.

A $v \times v$ \emph{Latin square} is a $v \times v$ array in which each row and each column
is a permutation of a $v$-set $Q$. A $k \times v$ array is called a \emph{Latin rectangle}
if each row is a permutation of $Q$ and each column contains $k$ distinct elements.
It is well known that such a rectangle can be completed to an $v \times v$ Latin square,
for example by applying the well-known Hall's marriage theorem~\cite{Hal35}. A Latin square has no $2 \times 2$ subsquares if any
$2 \times 2$ subsquare restricted to two rows and two columns does not form a Latin square~\cite{KLR75}.

A \emph{one-factorization} $F=\{F_0,F_1,\ldots,F_{v-2} \}$ of the complete graph $K_v$, where $v$ is an even positive integer,
is partition of the edges of $K_v$ into perfect matchings. Each $F_i$, $0 \leq i \leq v-2$, is a perfect matching of $K_v$,
which is also called a \emph{one-factor}. As mentioned before a one-factor is a Steiner system $S(1,2,v)$
and a one-factorization is a related large set.
Clearly, a one-factor contains $\frac{v}{2}$ pairs of vertices whose
union is the set of $v$ vertices in $K_v$.
\vspace{-0.2cm}

\subsection{DLS Construction}
\label{sec:DLS}

Let $(\Z_v,B)$ be an SQS($v$) and let $A$ be a $v \times v$ Latin square of order $v$, on the point set $\Z_v$, with
no $2 \times 2$ subsquares.
Denote by~$\alpha_i$ the permutation on $\Z_v$ defined by $\alpha_i(j)=y$ if and only if $A(i,j)=y$.
For each $i$, $0 \leq i \leq v-1$, we define a set of quadruples $B_i$ on $\Z_v \times \Z_2$ as follows:

\begin{enumerate}
\item For each quadruple $\{ x_1,x_2,x_3,x_4 \} \in B$, the following 8 quadruples are contained in~$B_i$.
$$
\{ (x_1,0),(x_2,0),(x_3,0),(\alpha_i(x_4),1) \},~~ \{ (x_1,1),(x_2,1),(x_3,1),(\alpha_i(x_4),0) \}
$$
$$
\{ (x_1,0),(x_2,0),(\alpha_i(x_3),1),(x_4,0) \},~~ \{ (x_1,1),(x_2,1),(\alpha_i(x_3),0),(x_4,1) \}
$$
$$
\{ (x_1,0),(\alpha_i(x_2),1),(x_3,0),(x_4,0) \},~~ \{ (x_1,1),(\alpha_i(x_2),0),(x_3,1),(x_4,1) \}
$$
$$
\{ (\alpha_i(x_1),1),(x_2,0),(x_3,0),(x_4,0) \},~~ \{ (\alpha_i(x_1),0),(x_2,1),(x_3,1),(x_4,1) \}
$$

\item For each pair $\{ x_1,x_2 \} \subset \Z_v$, the quadruple
$\{ (x_1,0),(x_2,0),(\alpha_i(x_1),1),(\alpha_i(x_2),1) \}$ is contained in~$B_i$.
\end{enumerate}

This DLS Construction is a variant~\cite{Etz96} of the Lindner Construction~\cite{Lin77}. It yields $v$ PDQs of order $2v$.
\vspace{-0.2cm}

\subsection{DB Construction}
\label{sec:DB}

Let $(\Z_v,B)$ be an SQS($v$), let $F=\{F_0,F_1,\ldots,F_{v-2} \}$ and
$F'=\{F'_0,F'_1,\ldots,F'_{n-2} \}$ be two one-factorizations of $K_v$ on the vertex set $\Z_v$, and
let $\alpha$ be any permutation on the set $\{ 0,1,\ldots,v-2 \}$. We define a collection of quadruples $B'$
on $\Z_v \times \Z_2$ as follows.

\begin{enumerate}
\item For each quadruple $\{ x_1,x_2,x_3,x_4 \} \in B$, the following two quadruples are
contained in~$B'$.
$$
\{ (x_1,0),(x_2,0),(x_3,0),(x_4,0) \},~~ \{ (x_1,1),(x_2,1),(x_3,1),(x_4,1) \}~.
$$

\item For each $i \in \{ 0,1,\ldots,n-2\}$ and $\{ x_1,x_2 \} \in F_i$, $\{ y_1,y_2 \} \in F'_j$,
where $j=\alpha (i)$, the following quadruple is contained in $B'$.
$$
\{ (x_1,0),(x_2,0),(y_1,1),(y_2,1) \}~.
$$
\end{enumerate}
The set of all such quadruples produced by $F_i$ and $F'_j$ are said to be the direct product of $F_i$ and~$F'_j$.
For the construction in Section~\ref{sec:recursion}, it is assumed that instead of one SQS($v$) $B$, we have
a set of $k$ PDQs of order $v$ and instead of $\alpha$ we
use $k$ permutations on $\{ 0,1,\ldots,v-2 \}$ defined by the rows of a $k \times (v-1)$ Latin rectangle.
In this case, the DB construction yields a set of $k$ PDQs of order $2v$.
\vspace{-0.2cm}

\subsection{Configurations}

Both constructions, the DLS Construction and the DB Construction, are based
on a partition of the point set $Q$ into two equal sets, say $P_1$ and $P_2$.
Given such a partition and a quadruple
$X = \{ x_1,x_2,x_3,x_4 \}$ of $Q$, we say that $X$ is a $4$-subset from \emph{configuration} $(i,j)$, where $i+j=4$, if
$|X \cap P_1|=i$ and $|X \cap P_2|=j$. There are clearly five possible configurations,
$(4,0)$, $(3,1)$, $(2,2)$, $(1,3)$, and $(0,4)$. The following trivial lemma can be easily verified.

\begin{lemma}
\label{lem:conf1_3}
In the DLS construction, based on the partition of $\Z_{v}\times\{0,1\}$ into $\Z_v \times \{0\}$ and $\Z_n \times \{1\}$,
each quadruple from configuration $(3,1)$ and each quadruple from configuration $(1,3)$
is contained exactly once in the $v$ PDQs of order $2v$.
\end{lemma}
\vspace{-0.2cm}

\section{An Improved Recursive Construction}
\label{sec:recursion}

In this section a recursive doubling construction for PDQs of order $4n$,
where $n$ is an odd integer not divisible by 3, is presented.
The construction proceeds in two parts, where we employ the two doubling constructions with $v=2n$, $n\equiv 1$ or 5 (mod 6). In the first part
the DLS Construction is applied with an appropriate $(2n) \times (2n)$ Latin square.
A set of $2n$ PDQs of order $4n$ is obtained in this part.
By Lemma~\ref{lem:conf1_3}, each quadruple from configurations $(3,1)$ and each quadruple
from configuration $(1,3)$ is contained in these $2n$ PDQs.
We analyse the quadruples from configuration $(2,2)$ contained in these $2n$ PDQs.
In the second part, the DB Construction is applied with a set of $k$ PDQs of order $2n$ and
a $k \times (2n-1)$ Latin rectangle which yields quadruples from configuration $(2,2)$
which are not obtained in the first part of the construction. The final outcome
from the two parts of the construction are $2n+k$ PDQs of order $4n$ which will imply Theorem~\ref{thm:main}.

\subsection{PDQs from the DLS Construction}
\label{sec:part1}

The first step in the first part of the construction is to define a $(2n) \times (2n)$ Latin square
to be used in the DLS Construction. This Latin square was defined in~\cite{KLR75} and used
in~\cite{Etz93} for constructions of PDQs of order $4n$.
Let $A_n$, $B_n$, be the two $n \times n$ Latin squares defined as follows.

$$
A_n (i,j) \equiv i+j-2~(\text{mod}~n), ~~~~ 1 \leq i,j \leq n,
$$
$$
B_n (i,j) \equiv i-j~(\text{mod}~n), ~~~~ 1 \leq i,j \leq n,
$$
where the values in $A_n$ and $B_n$ are reduced to the range $\{0,1,\ldots,n-1\}$.

Let $C_n$ and $D_n$ be the following two $n \times n$ Latin squares, where
for each $0 \leq i,j \leq n-1$, we have

$$
C_n(i,j) = A_n(i,j)~\text{reduced~modulo}~n~\text{to~the~range}~ \{n,n+1,\ldots,2n-1\},
$$
$$
D_n(i,j) = A_n(i,j)-1~\text{reduced~modulo}~n~\text{to~the~range}~\{n,n+1,\ldots,2n-1\}.
$$

The $(2n) \times (2n)$ Latin square $M_n$ which will be used in the DLS Construction is
given by
{\large
$$
M_n = \left[
\begin{array}{cc}
B_n & C_n \\
D_n & B_n
\end{array}
\right]
$$
}
It was proved in~\cite{KLR75} that $M_n$ has no $2 \times 2$ subsquares.

\begin{example}
\label{ex:M7}
Throughout this section, we will concentrate on a few examples, for $n=7$, $n=11$, an also $n=19$.

For $n=7$, the $14 \times 14$ Latin square $M_7$ is given by
$$
\left[
\begin{array}{ccccccccccccccc}
0 & 6 & 5 & 4 & 3 & 2 & 1 &  & 7 & 8 & 9 & 10 & 11 & 12 & 13\\
1 & 0 & 6 & 5 & 4 & 3 & 2 &  & 8 & 9 & 10 & 11 & 12 & 13 & 7\\
2 & 1 & 0 & 6 & 5 & 4 & 3 &  & 9 & 10 & 11 & 12 & 13 & 7 & 8\\
3 & 2 & 1 & 0 & 6 & 5 & 4 &  & 10 & 11 & 12 & 13 & 7 & 8 & 9\\
4 & 3 & 2 & 1 & 0 & 6 & 5 &  & 11 & 12 & 13 & 7 & 8 & 9 & 10\\
5 & 4 & 3 & 2 & 1 & 0 & 6 &  & 12 & 13 & 7 & 8 & 9 & 10 & 11\\
6 & 5 & 4 & 3 & 2 & 1 & 0 &  & 13 & 7 & 8 & 9 & 10 & 11 & 12\\ \\
13 & 7 & 8 & 9 & 10 & 11 & 12 &  & 0 & 6 & 5 & 4 & 3 & 2 & 1\\
7 & 8 & 9 & 10 & 11 & 12 & 13 &  & 1 & 0 & 6 & 5 & 4 & 3 & 2\\
8 & 9 & 10 & 11 & 12 & 13 & 7 &  & 2 & 1 & 0 & 6 & 5 & 4 & 3\\
9 & 10 & 11 & 12 & 13 & 7 & 8 &  & 3 & 2 & 1 & 0 & 6 & 5 & 4\\
10 & 11 & 12 & 13 & 7 & 8 & 9 &  & 4 & 3 & 2 & 1 & 0 & 6 & 5\\
11 & 12 & 13 & 7 & 8 & 9 & 10 &  & 5 & 4 & 3 & 2 & 1 & 0 & 6\\
12 & 13 & 7 & 8 & 9 & 10 & 11 &  & 6 & 5 & 4 & 3 & 2 & 1 & 0
\end{array}
\right]
$$

\end{example}

\subsection{PDQs from the DB Construction}
\label{sec:part2}

Now, we apply the DB Construction and obtain PDQs which are also disjoint
from the ones obtained in the first part via the DLS Construction. To this end,
we have first to analyse the quadruples from configuration $(2,2)$ which
are contained in the $2n$ PDQs obtained in the DLS Construction.
The reason is that by Lemma~\ref{lem:conf1_3} we can ignore quadruples from configuration
$(3,1)$ and configuration $(1,3)$ and clearly no quadruple from configurations
$(4,0)$ and $(0,4)$ is obtained in the DLS construction.

\subsubsection{Quadruples from Configuration $(2,2)$ in the DLS Construction}

The quadruples from configuration $(2,2)$ which are contained in the $2n$ PDQs of order $4n$
from the DLS Construction can be defined by observing the properties of
the Latin square $M_n$. First, note that four sets of pairs are defined.

$$
\cA_i \triangleq \{ \{ x,y \} ~:~ x,y \in \Z_n, ~ y - x \equiv i~(\text{mod}~n) \},~~ 1 \leq i \leq \frac{n-1}{2},
$$
$$
\cB_i \triangleq \{ \{ x,y \} ~:~ x,y \in \{n,n+1,\ldots,2n-1\}, ~ y - x \equiv i~(\text{mod}~n) \},~~ 1 \leq i \leq \frac{n-1}{2},
$$
$$
\cC_i \triangleq \{ \{x,y\}  ~:~  0 \leq x \leq n-1,~ n \leq y \leq 2n-1,~ x+y \equiv i~(\text{mod}~n)  \},~~0 \leq i \leq n-1,
$$
$$
\cD_i \triangleq \{ \{x,y\}  ~:~  0 \leq x \leq n-1,~ n \leq y \leq 2n-1,~ y-x \equiv i~(\text{mod}~n)  \},~~0 \leq i \leq n-1.
$$

\begin{example}
For $n=7$, the sets $\cA_i$'s, $\cB_i$'s, $\cC_i$'s, and $\cD_i$'s, are as follows.
$$
\cA_1 = \{ \{0,1\},\{1,2\},\{2,3\},\{3,4\},\{4,5\},\{5,6\},\{6,0\} \}
$$
$$
\cA_2 = \{ \{0,2\},\{1,3\},\{2,4\},\{3,5\},\{4,6\},\{5,0\},\{6,1\} \}
$$
$$
\cA_3 = \{ \{0,3\},\{1,4\},\{2,5\},\{3,6\},\{4,0\},\{5,1\},\{6,2\} \}
$$
\vspace{0.1cm}
$$
\cB_1 = \{ \{7,8\},\{8,9\},\{9,10\},\{10,11\},\{11,12\},\{12,13\},\{13,7\} \}
$$
$$
\cB_2 = \{ \{7,9\},\{8,10\},\{9,11\},\{10,12\},\{11,13\},\{12,7\},\{13,8\} \}
$$
$$
\cB_3 = \{ \{7,10\},\{8,11\},\{9,12\},\{10,13\},\{11,7\},\{12,8\},\{13,9\} \}
$$
\vspace{0.1cm}
$$
\cC_0 = \{ \{0,7\},\{1,13\},\{2,12\},\{3,11\},\{4,10\},\{5,9\},\{6,8\} \}
$$
$$
\cC_1 = \{ \{0,8\},\{1,7\},\{2,13\},\{3,12\},\{4,11\},\{5,10\},\{6,9\} \}
$$
$$
\cC_2 = \{ \{0,9\},\{1,8\},\{2,7\},\{3,13\},\{4,12\},\{5,11\},\{6,10\} \}
$$
$$
\cC_3 = \{ \{0,10\},\{1,9\},\{2,8\},\{3,7\},\{4,13\},\{5,12\},\{6,11\} \}
$$
$$
\cC_4 = \{ \{0,11\},\{1,10\},\{2,9\},\{3,8\},\{4,7\},\{5,13\},\{6,12\} \}
$$
$$
\cC_5 = \{ \{0,12\},\{1,11\},\{2,10\},\{3,9\},\{4,8\},\{5,7\},\{6,13\} \}
$$
$$
\cC_6 = \{ \{0,13\},\{1,12\},\{2,11\},\{3,10\},\{4,9\},\{5,8\},\{6,7\} \}
$$
\vspace{0.1cm}
$$
\cD_0 = \{ \{0,7\},\{1,8\},\{2,9\},\{3,10\},\{4,11\},\{5,12\},\{6,13\} \}
$$
$$
\cD_1 = \{ \{0,8\},\{1,9\},\{2,10\},\{3,11\},\{4,12\},\{5,13\},\{6,7\} \}
$$
$$
\cD_2 = \{ \{0,9\},\{1,10\},\{2,11\},\{3,12\},\{4,13\},\{5,7\},\{6,8\} \}
$$
$$
\cD_3 = \{ \{0,10\},\{1,11\},\{2,12\},\{3,13\},\{4,7\},\{5,8\},\{6,9\} \}
$$
$$
\cD_4 = \{ \{0,11\},\{1,12\},\{2,13\},\{3,7\},\{4,8\},\{5,9\},\{6,10\} \}
$$
$$
\cD_5 = \{ \{0,12\},\{1,13\},\{2,7\},\{3,8\},\{4,9\},\{5,10\},\{6,11\} \}
$$
$$
\cD_6 = \{ \{0,13\},\{1,7\},\{2,8\},\{3,9\},\{4,10\},\{5,11\},\{6,12\} \}
$$
\end{example}

The following two lemmas are immediate observations from the definitions of the
$\cA_i$'s, $\cB_i$'s, $\cC_i$'s, $\cD_i$'s, and $M_n$.

\begin{lemma}
\label{lem:pairsDLS}
By using the Latin square $M_n$, the following sets of quadruples are contained, in the first $2n$ PDQs
obtained via the DLS Construction.
$$
\{ (v,0),(x,0),(y,1),(z,1) \},~~\{v,x \} \in \cA_i \cup \cB_i,~~\{y,z \} \in \cA_i \cup \cB_i,~ 1 \leq i \leq \frac{n-1}{2},
$$
$$
\{ (v,0),(x,0),(y,1),(z,1) \},~~\{v,x \} \in \cC_i,~~\{y,z \} \in \cD_{i-1} \cup \cD_i,~ 0 \leq i \leq n-1,
$$
where subscripts are taken~ modulo $n$.
Each one of these quadruples is contained in exactly one of these $2n$ PDQs.
No other quadruple from configuration $(2,2)$ is contained in these $2n$ PDQs.
\end{lemma}

Note, that Lemma~\ref{lem:pairsDLS} implies that the quadruples from configuration $(2,2)$ formed by the
first $2n$ PDQs are of the form $\{ (v,0),(x,0),(y,1),(z,1) \}$.
The pair $\{ v,x \}$ is from one of the $2n-1$ sets $\cA_i$,
$\cB_i$, $1 \leq i \leq \frac{n-1}{2}$, or $\cC_j$, $0 \leq j \leq n-1$.
The pair $\{ y,z \}$ is from one of the $2n-1$ sets $\cA_i$,
$\cB_i$, $1 \leq i \leq \frac{n-1}{2}$, or $\cD_j$, $0 \leq j \leq n-1$.
These facts will be used in the sequel. Now, we state the second simple lemma which can be easily verified.

\begin{lemma}
\label{lem:one-factors}
$~$
\begin{enumerate}
\item Each one of the $\cC_i$'s and the $\cD_i$'s is a one-factor of $K_{2n}$.


\item $\bigcup_{i=1}^{(n-1)/2} (\cA_i \cup \cB_i) \cup \bigcup_{i=0}^{n-1} \cC_i = \{ \{x,y\} ~:~ 0 \leq x<y \leq 2n-1 \}$.

\item $\bigcup_{i=1}^{(n-1)/2} (\cA_i \cup \cB_i) \cup \bigcup_{i=0}^{n-1} \cD_i = \{ \{x,y\} ~:~ 0 \leq x<y \leq 2n-1 \}$.

\item $\sum_{i=1}^{(n-1)/2} (|\cA_i| +|\cB_i|)+ \sum_{i=0}^{n-1} |\cC_i| = n(2n-1)$.

\item $\sum_{i=1}^{(n-1)/2} (|\cA_i| +|\cB_i|)+ \sum_{i=0}^{n-1} |\cD_i| = n(2n-1)$.
\end{enumerate}
\end{lemma}

\subsubsection{One-Factorizations for the DB Construction}

The second part of the construction is to use the DB Construction to form more PDQs or order $4n$ with quadruples which were not
used in the first $2n$ PDQs. Such a set of new PDQs of order $4n$ depends on the value of $D(2n)$.
These PDQs of order $2n$ will form the quadruples from configuration $(4,0)$
and from configuration $(0,4)$ in the recursive construction. We distinguish now between the case
$n \equiv 5~(\text{mod}~6)$ and the case
$n \equiv 1~(\text{mod}~6)$.

\begin{lemma}
\label{lem:primesR}
$~$
Let $n\equiv 1$ or $5~(\text{mod}~6)$ be a prime. Then  there exists a one-factorization $\cR$ of $K_{2n}$
on the vertex set $\{0,1,\ldots, 2n-1\}$ satisfying the followings:
\begin{enumerate}
\item  if $n\equiv 5~(\text{mod}~6)$, then ${\cal R}=\{{\cal R}_{i,j}:1\le i\le {2n-1\over 3}, ~ 1 \leq j \leq 3 \},$ where
\begin{align}\label{redistr}{\cal R}_{i,1}\cup {\cal R}_{i,2}\cup {\cal R}_{i,3}={\cal A}_i\cup {\cal B}_i\cup{\cal C}_{i-1},1\le i\le {n-1\over 2},\end{align}
  and $${\cR}_{i,1}={\cC}_{3i-n-2}, {\cal R}_{i,2}= {\cC}_{3i-n-1}, {\cal R}_{i,3}={\cal C}_{3i-n},{n+1\over 2}\le i\le {2n-1\over 3}.$$

\item  if $n\equiv 1~(\text{mod}~6)$, then ${\cal R}=\{{\cal R}_{i,j}:1\le i\le {n-1\over 2}, ~ 1 \leq j \leq 3\}\cup\{{\cal C}_i:{n-1\over2}\le i\le n-1\}$, where Eq. {\rm (\ref{redistr})} also holds.
\end{enumerate}
\end{lemma}
\begin{proof}
By Lemma~\ref{lem:one-factors}, each ${\cC}_i$ is a one-factor of $K_{2n}$ and $\bigcup_{i=1}^{n-1\over 2}({\cal A}_i\cup{\cal B}_i)\cup(\bigcup_{i=0}^{n-1}{\cal C}_i)$ covers all pairs of $\{0,1,\ldots, 2n-1\}$. So, to prove the conclusion, we only need to show that there exist one-factors ${\cal R}_{i,j}$ with $1\le i\le {2n-1\over 3}$ and $j=1,2,3$ such that Eq. (\ref{redistr}) holds.

For a given $1\le i\le {n-1\over 2}$, form the sequences $s_j =\{j\cdot i, (j + 1) \cdot i\}, 0\le j \le n - 1$,
where the elements are taken modulo $n$, reduced to the range $\{0, 1, \ldots , n -1\}$.
Form the second sequence $q_j =\{ j \cdot i, (j +1) \cdot i\}, 0\le j \le n - 1$, where elements are
taken modulo $n$, reduced to the range $\{n, n + 1, \ldots , 2n -1\}$. It is obvious that ${\cA}_i=\{s_j:0\le j\le n-1\}$
and ${\cB}_i=\{q_j:0\le j\le n-1\}$. Since $n$ is prime, the $n$ pairs of ${\cA}_i$ form an $n$-cycle (a cycle of length $n$),
so are those of ${\cB}_i$. Note that
$s_0=\{0,i\}\in {\cA}_i,q_{m}=\{n+i-1,2n-1\}\in {\cB}_i$ where $n-1\equiv m\cdot i$ (mod $n$) ($0\le m\le n-1$).
Since $n$ is odd, applying the $n$-cycle of ${\cA}_i$ we can partition
${\cA}_i\setminus\{s_0\}$ into two subsets ${\cA}_{i,1}$ and  ${\cA}_{i,2}$ with equal size ${n-1\over 2}$ and
$$
\bigcup_{P\in{\cA}_{i,1}}P=\{0,1,\ldots,n-1\}\setminus\{0\},\bigcup_{P\in{\cA}_{i,2}}P=\{0,1,\ldots,n-1\}\setminus\{i\}.
$$
Similarly, we can partition ${\cal B}_i\setminus\{q_m\}$ into two subsets
${\cB}_{i,1}$ and  ${\cal B}_{i,2}$ with equal size ${n-1\over 2}$ and
$$
\bigcup_{P\in{\cal B}_{i,1}}P=\{n,n+1,\ldots,2n-1\}\setminus\{n+i-1\},\bigcup_{P\in{\cB}_{i,2}}P=\{n,n+1,\ldots,2n-1\}\setminus\{2n-1\}.
$$
Obviously we have
$\{0,n+i-1\},\{i,2n-1\}\in{\cC}_{i-1}$. So define
\begin{align*}
&{\cal R}_{i,1}={\cal A}_{i,1}\cup{\cal B}_{i,1}\cup\{\{0,n+i-1\}\},\\
&{\cal R}_{i,2}={\cal A}_{i,2}\cup{\cal B}_{i,2}\cup\{\{i,2n-1\}\},\\
&{\cal R}_{i,3}=({\cal C}_{i-1}\cup\{s_0,q_m\})\setminus\{\{0,n+i-1\},\{i,2n-1\}\}.
\end{align*}
It is easy to verify that each ${\cal R}_{i,j}$ is a one-factor ($1\le i\le {2n-1\over 3}$ and $1 \leq j \leq 3$) and Eq. (\ref{redistr}) holds. This completes the proof.
\end{proof}
A related result will be given in the following lemma for a one-factorization $\cT$ with almost an identical proof.
\begin{lemma}
\label{lem:primesT}
Let $n\equiv 1$ or $5~(\text{mod}~6)$ be a prime. Then there exists a one-factorization $\cT$ of $K_{2n}$
on the vertex set $\{0,1,\ldots, 2n-1\}$ satisfying the followings:
\begin{enumerate}
\item if $n\equiv 5$ {\rm (mod 6)}, then ${\cal T}=\{{\cal T}_{i,j}:1\le i\le {2n-1\over 3}~ 1 \leq j \leq 3\},$ where
\begin{align}\label{redistr2}{\cal T}_{i,1}\cup {\cal T}_{i,2}\cup {\cal T}_{i,3}={\cal A}_i\cup {\cal B}_i\cup{\cal D}_{i-1},1\le i\le {n-1\over 2},\end{align}
  and $${\cal T}_{i,1}={\cal D}_{3i-n-2}, {\cal T}_{i,2}= {\cal D}_{3i-n-1}, {\cal T}_{i,3}={\cal D}_{3i-n},{n+1\over 2}\le i\le {2n-1\over 3}.$$

\item   if $n\equiv 1$ {\rm (mod 6)}, then ${\cal T}=\{{\cal T}_{i,j}:1\le i\le {n-1\over 2~ 1 \leq j \leq 3 }\}\cup\{{\cal D}_i:{n-1\over2}\le i\le n-1\},$ where Eq. {\rm (\ref{redistr2})} also holds.
\end{enumerate}
\end{lemma}
\begin{proof}
The proof is very similar to that of Lemma~\ref{lem:primesR}. The difference is that~$\cD_i$
is used instead of~${\cC}_i$. To prove that, for $1\le i\le {n-1\over 2}$, ${\cal A}_i\cup {\cB}_i\cup{\cD}_{i-1}$
can be partitioned into three one-factors ${\cT}_{i,1}, {\cT}_{i,2}$, and~${\cT}_{i,3}$, we note
that $\{0,i\}\in{\cA}_i,\{n+i-1,n+2i-1\}\in{\cB}_i$ and $\{0,n+i-1\},\{i,n+2i-1\}\in{\cD}_{i-1}$. The rest of
the proof is exactly parallel to that of Lemma~\ref{lem:primesR}.
\end{proof}

\begin{example}
For $n=11$, the set $\cR$ and part of the set $\cT$ are given by the following sets.
$$
\cR_{1,1} = \{ \{1,2\},\{3,4\},\{5,6\},\{7,8\},\{9,10\},\{12,13\},\{14,15\},\{16,17\},\{18,19\},\{20,21\},\{0,11\}  \}
$$
$$
\cR_{1,2} = \{ \{2,3\},\{4,5\},\{6,7\},\{8,9\},\{10,0\},\{11,12\},\{13,14\},\{15,16\},\{17,18\},\{19,20\},\{1,21\}  \}
$$
$$
\cR_{1,3} = \{ \{2,20\},\{3,19\},\{4,18\},\{5,17\},\{6,16\},\{7,15\},\{8,14\},\{9,13\},\{10,12\},\{0,1\},\{11,21\}  \}
$$
$$
\cR_{2,1} = \{ \{2,4\},\{3,5\},\{6,8\},\{7,9\},\{1,10\},\{11,13\},\{14,16\},\{15,17\},\{18,20\},\{19,21\},\{0,12\}  \}
$$
$$
\cR_{2,2} = \{ \{1,3\},\{4,6\},\{5,7\},\{8,10\},\{0,9\},\{12,14\},\{13,15\},\{16,18\},\{17,19\},\{11,20\},\{2,21\}  \}
$$
$$
\cR_{2,3} = \{ \{1,11\},\{3,20\},\{4,19\},\{5,18\},\{6,17\},\{7,16\},\{8,15\},\{9,14\},\{10,13\},\{0,2\},\{12,21\}  \}
$$
$$
\cR_{3,1} = \{ \{3,6\},\{4,7\},\{5,8\},\{1,9\},\{2,10\},\{11,14\},\{12,15\},\{16,19\},\{17,20\},\{18,21\},\{0,13\}  \}
$$
$$
\cR_{3,2} = \{ \{1,4\},\{2,5\},\{6,9\},\{7,10\},\{8,0\},\{13,16\},\{14,17\},\{15,18\},\{11,19\},\{12,20\},\{3,21\}  \}
$$
$$
\cR_{3,3} = \{ \{1,12\},\{2,11\},\{4,20\},\{5,19\},\{6,18\},\{7,17\},\{8,16\},\{9,15\},\{10,14\},\{0,3\},\{13,21\}  \}
$$
$$
\cR_{4,1} = \{ \{1,5\},\{3,7\},\{4,8\},\{6,10\},\{2,9\},\{12,16\},\{15,19\},\{17,21\},\{11,18\},\{13,20\},\{0,14\}  \}
$$
$$
\cR_{4,2} = \{ \{2,6\},\{5,9\},\{0,7\},\{1,8\},\{3,10\},\{11,15\},\{13,17\},\{14,18\},\{16,20\},\{12,19\},\{4,21\}  \}
$$
$$
\cR_{4,3} = \{ \{1,13\},\{2,12\},\{3,11\},\{5,20\},\{6,19\},\{7,18\},\{8,17\},\{9,16\},\{10,15\},\{0,4\},\{14,21\}  \}
$$
$$
\cR_{5,1} = \{ \{1,6\},\{2,7\},\{3,8\},\{4,9\},\{5,10\},\{16,21\},\{11,17\},\{12,18\},\{13,19\},\{14,20\},\{0,15\}  \}
$$
$$
\cR_{5,2} = \{ \{0,6\},\{1,7\},\{2,8\},\{3,9\},\{4,10\},\{11,16\},\{12,17\},\{13,18\},\{14,19\},\{15,20\},\{5,21\}  \}
$$
$$
\cR_{5,3} = \{ \{1,14\},\{2,13\},\{3,12\},\{4,11\},\{6,20\},\{7,19\},\{8,18\},\{9,17\},\{10,16\},\{0,5\},\{15,21\}  \}
$$
$$
\cR_{6,1} = \cC_5,~  \cR_{6,2} = \cC_6,~ \cR_{6,3} = \cC_7,
$$
$$
\cR_{7,1} = \cC_8,~  \cR_{7,2} = \cC_9,~ \cR_{7,3} = \cC_{10},
$$
\vspace{0.1cm}
$$
\cT_{1,1} = \{ \{1,2\},\{3,4\},\{5,6\},\{7,8\},\{9,10\},\{12,13\},\{14,15\},\{16,17\},\{18,19\},\{20,21\},\{0,11\}  \}
$$
$$
\cT_{1,2} = \{ \{2,3\},\{4,5\},\{6,7\},\{8,9\},\{0,10\},\{11,21\},\{13,14\},\{15,16\},\{17,18\},\{19,20\},\{1,12\}  \}
$$
$$
\cT_{1,3} = \{ \{2,13\},\{3,14\},\{4,15\},\{5,16\},\{6,17\},\{7,18\},\{8,19\},\{9,20\},\{10,21\},\{0,1\},\{11,12\}  \}
$$
$$
\cT_{2,1} = \{ \{2,4\},\{3,5\},\{6,8\},\{7,9\},\{1,10\},\{11,13\},\{14,16\},\{15,17\},\{18,20\},\{19,21\},\{0,12\}  \}
$$
$$
\cT_{2,2} = \{ \{1,3\},\{4,6\},\{5,7\},\{8,10\},\{0,9\},\{13,15\},\{16,18\},\{17,19\},\{11,20\},\{12,21\},\{2,14\}  \}
$$
$$
\cT_{2,3} = \{ \{1,13\},\{3,15\},\{4,16\},\{5,17\},\{6,18\},\{7,19\},\{8,20\},\{9,21\},\{10,11\},\{0,2\},\{12,14\}  \}
$$
\end{example}

\vspace{0.2cm}

It is not difficult to accommodate the proofs of Lemmas~\ref{lem:primesR} and~\ref{lem:primesT}
to all positive integers $n \equiv 1$ or $5~(\text{mod}~6)$, where $n \geq 7$.
\begin{lemma}
\label{lem:generalRT}
Lemmas~\ref{lem:primesR} and~\ref{lem:primesT} hold for any positive integer $n \equiv 1$ or $5~(\text{mod}~6)$, where $n \geq 7$.
\end{lemma}
\begin{proof}
When $n \equiv 1$ or $5~(\text{mod}~6)$,
the proofs for Lemmas~\ref{lem:primesR} and~\ref{lem:primesT} are similar, so we only show the existence of
a one-factorization ${\cT}$ for Lemma~\ref{lem:primesT}. The proof is actually similar to the case of $n$ being prime.
The difference  is that  ${\cA}_i$ and ${\cB}_i$
may consist of  more than one cycle if $n$ is not a prime.
Nevertheless, we can use the same procedure. More precisely, we are done as follows.

Assume that we want to partition ${\cA}_i\cup {\cB}_i\cup{\cD}_{i-1},1\le i\le {n-1\over 2}$, into three one-factors.
The length of the cycles  generated from $\cA_i$ is $\ell= {n\over \gcd(i,n)}$ and the number
of cycles is $\rho=\gcd(i,n)$. We denote the $j$-th cycle, $0\le j\le \rho-1$,
by ${\cal A}_{i}^{(j)}$, which consists of the following $\ell$ edges (elements are reduced to the range $\{0,1,\ldots,n-1\}$):
$$\{j, i + j\}, \{i + j, 2i + j\}, \{2i + j, 3i + j\}, \ldots, \{n- i + j, j\}.$$
For $\cB_i$ we have the same $\rho$ cycles of length $\ell$, say ${\cal B}_{i}^{(j)},0\le j\le \rho-1$,
where the elements are reduced to the range $\{n,n+1,\ldots,2n-1\}$. Note that $\ell$ is odd
and $\{j,i+j\}\in{\cal A}_i^{(j)},\{\beta_{i,j},b_{i,j}\}\in \cB_i^{(j')}$,
where $\beta_{i,j} \equiv i+j-1$ (mod $n$), $b_{i,j}\equiv 2i+j-1$ (mod $n$) with $n\le \beta_{i,j},b_{i,j}\le 2n-1$,
and $j'\equiv j-1$ (mod $\rho$) with $0\le j'\le \rho-1$. We can partition ${\cal A}_i^{(j)}\setminus\{\{j,i+j\}\}$
into two subsets  ${\cal A}_{i,1}^{(j)}$, and ${\cal A}_{i,2}^{(j)}$ with equal
size ${\ell-1\over 2}$ and
$$
\bigcup_{P\in \cA_{i,1}^{(j)}}P=(\bigcup_{P\in \cA_{i}^{(j)}}P)\setminus\{j\},\bigcup_{P\in \cA_{i,2}^{(j)}}P=(\bigcup_{P\in \cA_{i}^{(j)}}P)\setminus\{i+j\}.
$$
Similarly, we can partition  $\cB_i^{(j')}\setminus\{\{\beta_{i,j},b_{i,j}\}\}$ into two subsets
$\cB_{i,1}^{(j')}$ and $\cB_{i,2}^{(j')}$ with equal size ${\ell-1\over 2}$ and
$$
\bigcup_{P\in \cB_{i,1}^{(j')}}P=(\bigcup_{P\in \cB_{i}^{(j')}}P)\setminus\{\beta_{i,j}\},\bigcup_{P\in \cB_{i,2}^{(j')}}P=(\bigcup_{P\in \cB_{i}^{(j')}}P)\setminus\{b_{i,j}\}.
$$
Obviously we  have $\{j,\beta_{i,j}\},\{i+j,b_{i,j}\}\in{\cal D}_{i-1}$. So define
\begin{align*}
&\cT_{i,1}=\bigcup_{j=0}^{\rho-1}(\cA_{i,1}^{(j)}\cup \cB_{i,1}^{(j')}\cup\{\{j,\beta_{i,j}\}\}),\\
&\cT_{i,2}=\bigcup_{j=0}^{\rho-1}(\cA_{i,2}^{(j)}\cup{\cal B}_{i,2}^{(j')}\cup\{\{i+j,b_{i,j}\}\}),\\
&\cT_{i,3}=({\cal D}_{i-1}\cup\{\{j,i+j\},\{\beta_{i,j},b_{i,j}\}:0\le j\le \rho-1\})\\
&\qquad\qquad \ \ \ \ \setminus\{\{j,\beta_{i,j}\},\{i+j,b_{i,j}\}:0\le j\le \rho-1\}.
\end{align*}
It is easy to verify that each $\cT_{i,k}$ is a one-factor ($1\le i\le {n-1\over 2}$ and $k=1,2,3$) and Eq. (\ref{redistr2}) holds.
\end{proof}

\subsubsection{A Latin Square for Direct Products}
\label{sec:LatinS}

This subsection is devoted only to the case when $n \equiv 5~(\text{mod}~6)$.
Let $\cR_i \triangleq \{ \cR_{i,1},\cR_{i,2},\cR_{i,3} \}$ and
$\cT_i \triangleq \{ \cT_{i,1},\cT_{i,2},\cT_{i,3} \}$, $1 \leq i \leq \frac{2n-1}{3}$.
The same definition for $1 \leq i \leq \frac{n-1}{2}$ is given for $n \equiv 1~(\text{mod}~6)$.
We construct a Latin square which makes a distinction between
the the quadruples which are contained in
the DLS Construction and those which will be constructed with the DB Construction.

Let $\nu = \frac{2n-1}{3} $ and let $\cM$ be
a $\nu \times \nu$ Latin square defined as follows.
The elements of $\cM$ are taken from the set $\{\cT_1,\cT_2,\ldots,\cT_\nu \}$.
The $i$-th column of $\cM$ is indexed by $\cR_i$, $1 \leq i \leq \nu$.
The first two rows of $\cM$ represent direct products of $\cR_i$'s by $\cT_j$'s
for which some of the related quadruples are contained in the first $2n$ PDQs of order $4n$,
and all the quadruples of the first $2n$ PDQs are contained in these direct products, which will be explained in the next subsection.
For this purpose let $\cM(1,\cR_i) =\cT_i$ for $1 \leq i \leq \nu$,
$\cM(2,\cR_i) = \cT_{i-1}$ for $2 \leq i \leq \nu$, and $\cM(2,\cR_1)=\cT_\nu$.
The last $\nu -2$ rows of $\cM$ can be chosen arbitrarily to form a Latin square.
One simple choice is to take these rows as cyclic shifts of the first row.

\begin{example}
\label{ex:LS_DB}
We continue the example for $n=11$.
The first two rows of $\cM$, represent the quadruples from configuration $(2,2)$
which are contained in the first 22 PDQs of order 44 obtained by the DLS Construction.
These rows are given by
$$
\begin{array}{lccccccc}
& \cR_1 & \cR_2 & \cR_3 & \cR_4 & \cR_5 & \cR_6 & \cR_7 \\ \hline
1 & \cT_1 & \cT_2 & \cT_3 & \cT_4 & \cT_5 & \cT_6 & \cT_7 \\
2 & \cT_7 & \cT_1 & \cT_2 & \cT_3 & \cT_4 & \cT_5 & \cT_6
\end{array}
$$
These two rows can be completed to a $7 \times 7$ Latin square as follows
$$
\begin{array}{lccccccc}
& \cR_1 & \cR_2 & \cR_3 & \cR_4 & \cR_5 & \cR_6 & \cR_7 \\ \hline
1 & \cT_1 & \cT_2 & \cT_3 & \cT_4 & \cT_5 & \cT_6 & \cT_7 \\
2 & \cT_7 & \cT_1 & \cT_2 & \cT_3 & \cT_4 & \cT_5 & \cT_6 \\ \hline
3 & \cT_6 & \cT_7 & \cT_1 & \cT_2 & \cT_3 & \cT_4 & \cT_5 \\
4 & \cT_5 & \cT_6 & \cT_7 & \cT_1 & \cT_2 & \cT_3 & \cT_4 \\
5 & \cT_4 & \cT_5 & \cT_6 & \cT_7 & \cT_1 & \cT_2 & \cT_3 \\
6 & \cT_3 & \cT_4 & \cT_5 & \cT_6 & \cT_7 & \cT_1 & \cT_2 \\
7 & \cT_2 & \cT_3 & \cT_4 & \cT_5 & \cT_6 & \cT_7 & \cT_1
\end{array}.
$$

\end{example}

\subsubsection{Proof of Theorem~\ref{thm:main}}

We distinguish again between the case
$n \equiv 5~(\text{mod}~6)$ and the case $n \equiv 1~(\text{mod}~6)$.

\noindent
{\bf Case 1:} $n \equiv 5~(\text{mod}~6)$.

We continue now to conclude the second part of the construction and obtain more
PDQs of order $4n$. We construct a $(2n-1) \times (2n-1)$ Latin square $\cH_{2n-1}$ from the Latin square $\cM$
as follows. Each row $i$ of $\cM$ is replaced by rows $(i,1)$, $(i,2)$, and $(i,3)$, in $\cH_{2n-1}$.
Each column $\cR_j$ of $\cM$ is replaced by columns $\cR_{j,1}$,
$\cR_{j,2}$, and $\cR_{j,3}$, in $\cH_{2n-1}$.
Any entry of $\cM(i,\cR_j)$ which is equal to $\cT_\ell$ is replaced by the $3 \times 3$ Latin square in the related
three rows and three columns of $\cH_{2n-1}$:

$$
\begin{array}{lcccc}
& \vline & \cR_{j,1} & \cR_{j,2} & \cR_{j,3}  \\ \hline
(i,1) & \vline & \cT_{\ell,1} & \cT_{\ell,2} & \cT_{\ell,3}  \\
(i,2) & \vline & \cT_{\ell,2} & \cT_{\ell,3} & \cT_{\ell,1}  \\
(i,3) & \vline & \cT_{\ell,3} & \cT_{\ell,1} & \cT_{\ell,2}
\end{array}.
$$

%

\noindent
Then each row of ${\cal H}_{2n-1}$ forms the one-factorization ${\cal T}$.
By Lemma~\ref{lem:pairsDLS}, all the quadruples from configuration $(2,2)$ in the first $2n$ PDQs from the DLS Construction are
contained in the direct products related to the first six rows of $\cH_{2n-1}$.
That is, the direct products of ${\cal R}_{j,k}$ and $\cH_{2n-1}((i,k'),\cR_{j,k})$
for all $1\leq j\le {2n-1\over 3}$, $i=1,2$, and $1\le k,k'\le 3$.
All the direct products related to the last $2n-7$ rows of $\cH_{2n-1}$
are pairwise disjoint and hence by using them in the DB Construction, combined with $k$
PDQs of order $2n$ in the DB Construction, where $k = \min \{ D(2n),2n-7\}$
we obtain with the first $2n$ PDQs,
a total of $2n +\min \{ D(2n),2n-7\}$ PDQs of order $4n$. Thus, we have proved
Theorem~\ref{thm:main} for $n \equiv 5~(\text{mod}~6)$.

\begin{example}
\label{ex:DB11}
Continue with the example for $n=11$. We first form the first six rows of ${\cal H}_{21}$
from the Latin square ${\cal M}$ in Example~\ref{ex:LS_DB}. If column $\cR_{i_1,j_1}$, where $1 \leq i_1 \leq 6$,
has entry $\cT_{i_2,j_2}$ then the DLS
construction has used some quadruples $\{ (v,0),(x,1),(y,2),(z,3) \}$, where $\{ v,x \} \in \cR_{i_1,j_1}$
and $\{ y,z \} \in \cT_{i_2,j_2}$.
Clearly, there are also many of the quadruples related to these six rows were not obtained in the first part of the construction.
Finally, note that instead of some $\cR_{i,j}$'s and $\cT_{i,j}$'s we use $\cC_{\ell}$'s and $\cD_{\ell}$'s.

{\fontsize{4.5}{12} \selectfont
$$
\begin{array}{ccccccccccccccccccccc}
\cR_{1,1} & \cR_{1,2} & \cR_{1,3} & \cR_{2,1} & \cR_{2,2} & \cR_{2,3} & \cR_{3,1} & \cR_{3,2} & \cR_{3,3} &
\cR_{4,1} & \cR_{4,2} & \cR_{4,3} & \cR_{5,1} & \cR_{5,2} & \cR_{5,3} & \cC_5 & \cC_6 & \cC_7 & \cC_8 & \cC_9 & \cC_{10}\\ \hline
\cT_{1,1} & \cT_{1,2} & \cT_{1,3} & \cT_{2,1} & \cT_{2,2} & \cT_{2,3} & \cT_{3,1} & \cT_{3,2} & \cT_{3,3} &
\cT_{4,1} & \cT_{4,2} & \cT_{4,3} & \cT_{5,1} & \cT_{5,2} & \cT_{5,3} & \cD_5 & \cD_6 & \cD_7 & \cD_8 & \cD_9 & \cD_{10} \\
\cT_{1,2} & \cT_{1,3} & \cT_{1,1} & \cT_{2,2} & \cT_{2,3} & \cT_{2,1} & \cT_{3,2} & \cT_{3,3} & \cT_{3,1} &
\cT_{4,2} & \cT_{4,3} & \cT_{4,1} & \cT_{5,2} & \cT_{5,3} & \cT_{5,1} & \cD_6 & \cD_7 & \cD_5 & \cD_9 & \cD_{10} & \cD_8 \\
\cT_{1,3} & \cT_{1,1} & \cT_{1,2} & \cT_{2,3} & \cT_{2,1} & \cT_{2,2} & \cT_{3,3} & \cT_{3,1} & \cT_{3,2} &
\cT_{4,3} & \cT_{4,1} & \cT_{4,2} & \cT_{5,3} & \cT_{5,1} & \cT_{5,2} & \cD_7 & \cD_5 & \cD_6 & \cD_{10} & \cD_8 & \cD_9 \\
\cD_8 & \cD_9 & \cD_{10} & \cT_{1,1} & \cT_{1,2} & \cT_{1,3} & \cT_{2,1} & \cT_{2,2} & \cT_{2,3} &
\cT_{3,1} & \cT_{3,2} & \cT_{3,3} & \cT_{4,1} & \cT_{4,2} & \cT_{4,3} & \cT_{5,1} & \cT_{5,2} & \cT_{5,3} & \cD_5 & \cD_6 & \cD_7 \\
\cD_9 & \cD_{10} & \cD_8 & \cT_{1,2} & \cT_{1,3} & \cT_{1,1} & \cT_{2,2} & \cT_{2,3} & \cT_{2,1} &
\cT_{3,2} & \cT_{3,3} & \cT_{3,1} & \cT_{4,2} & \cT_{4,3} & \cT_{4,1} & \cT_{5,2} & \cT_{5,3} & \cT_{5,1} & \cD_6 & \cD_7 & \cD_5 \\
\cD_{10} & \cD_8 & \cD_9 & \cT_{1,3} & \cT_{1,1} & \cT_{1,2} & \cT_{2,3} & \cT_{2,1} & \cT_{2,2} &
\cT_{3,3} & \cT_{3,1} & \cT_{3,2} & \cT_{4,3} & \cT_{4,1} & \cT_{4,2} & \cT_{5,3} & \cT_{5,1} & \cT_{5,2} & \cD_7 & \cD_5 & \cD_6 \\
\end{array}
$$}

A completed $21 \times 21$ Latin square $\cH_{21}$ whose last 15 rows form the $15 \times 21$ Latin rectangle
for the DB Construction is given next. In the first 6 rows we remove those entries for which
no quadruple of the related direct product is contained in the PDQs obtained
by the DLS Construction. Note, that each column contains different elements and each one of the
last 15 rows contains each one-factor of $\cT$ exactly once.

$$
\begin{array}{cccccccccccccc}
& \cR_{1,1} & \cR_{1,2} & \cR_{1,3} & \cdots & \cR_{5,1} & \cR_{5,2} & \cR_{5,3} & \cC_5 & \cC_6 & \cC_7 & \cC_8 & \cC_9 & \cC_{10}\\ \hline
1 & \cT_{1,1} & \cT_{1,2} & \cT_{1,3} & \cdots & \cT_{5,1} & \cT_{5,2} & \cT_{5,3} & \cD_5 & \cD_6 & \cD_7 & \cD_8 & \cD_9 & \cD_{10} \\
2 & \cT_{1,2} & \cT_{1,3} & \cT_{1,1} & \cdots & \cT_{5,2} & \cT_{5,3} & \cT_{5,1} &  &  &  &  &  &  \\
3 & \cT_{1,3} & \cT_{1,1} & \cT_{1,2} & \cdots & \cT_{5,3} & \cT_{5,1} & \cT_{5,2} &  & \cD_5 & \cD_6 & & \cD_8 & \cD_9 \\
4 & &  & \cD_{10} & \cdots & \cT_{4,1} & \cT_{4,2} & \cT_{4,3} & \cT_{5,1} &  &  &  &  &  \\
5 & & \cD_{10} &  & \cdots & \cT_{4,2} & \cT_{4,3} & \cT_{4,1} & \cT_{5,2} &  &  & &  &  \\
6 & \cD_{10} & & & \cdots & \cT_{4,3} & \cT_{4,1} & \cT_{4,2} & \cT_{5,3} &  &  & \cD_7 &  &  \\ \hline
7 & \cD_5 & \cD_6 & \cD_7 & \cdots & \cT_{3,1} & \cT_{3,2} & \cT_{3,3} & \cT_{4,1} & \cT_{4,2} & \cT_{4,3} & \cT_{5,1} & \cT_{5,2} & \cT_{5,3} \\
8 & \cD_6 & \cD_7 & \cD_5 & \cdots & \cT_{3,2} & \cT_{3,3} & \cT_{3,1} & \cT_{4,2} & \cT_{4,3} & \cT_{4,1} & \cT_{5,2} & \cT_{5,3} & \cT_{5,1} \\
9 & \cD_7 & \cD_5 & \cD_6 & \cdots & \cT_{3,3} & \cT_{3,1} & \cT_{3,2} & \cT_{4,3} & \cT_{4,1} & \cT_{4,2} & \cT_{5,3} & \cT_{5,1} & \cT_{5,2} \\
10 & \cT_{5,1} & \cT_{5,2} & \cT_{5,3} & \cdots & \cT_{2,1} & \cT_{2,2} & \cT_{2,3} & \cT_{3,1} & \cT_{3,2} & \cT_{3,3} & \cT_{4,1} & \cT_{4,2} & \cT_{4,3} \\
11 & \cT_{5,2} & \cT_{5,3} & \cT_{5,1} & \cdots & \cT_{2,2} & \cT_{2,3} & \cT_{2,1} & \cT_{3,2} & \cT_{3,3} & \cT_{3,1} & \cT_{4,2} & \cT_{4,3} & \cT_{4,1} \\
12 & \cT_{5,3} & \cT_{5,1} & \cT_{5,2} & \cdots & \cT_{2,3} & \cT_{2,1} & \cT_{2,2} & \cT_{3,3} & \cT_{3,1} & \cT_{3,2} & \cT_{4,3} & \cT_{4,1} & \cT_{4,2} \\
13 & \cT_{4,1} & \cT_{4,2} & \cT_{4,3} & \cdots & \cT_{1,1} & \cT_{1,2} & \cT_{1,3} & \cT_{2,1} & \cT_{2,2} & \cT_{2,3} & \cT_{3,1} & \cT_{3,2} & \cT_{3,3} \\
14 & \cT_{4,2} & \cT_{4,3} & \cT_{4,1} & \cdots & \cT_{1,2} & \cT_{1,3} & \cT_{1,1} & \cT_{2,2} & \cT_{2,3} & \cT_{2,1} & \cT_{3,2} & \cT_{3,3} & \cT_{3,1} \\
15 & \cT_{4,3} & \cT_{4,1} & \cT_{4,2} & \cdots & \cT_{1,3} & \cT_{1,1} & \cT_{1,2} & \cT_{2,3} & \cT_{2,1} & \cT_{2,2} & \cT_{3,3} & \cT_{3,1} & \cT_{3,2} \\
16 & \cT_{3,1} & \cT_{3,2} & \cT_{3,3} & \cdots & \cD_8 & \cD_9 & \cD_{10} & \cT_{1,1} & \cT_{1,2} & \cT_{1,3} & \cT_{2,1} & \cT_{2,2} & \cT_{2,3} \\
17 & \cT_{3,2} & \cT_{3,3} & \cT_{3,1} & \cdots & \cD_9 & \cD_{10} & \cD_8 & \cT_{1,2} & \cT_{1,3} & \cT_{1,1} & \cT_{2,2} & \cT_{2,3} & \cT_{2,1} \\
18 & \cT_{3,3} & \cT_{3,1} & \cT_{3,2} & \cdots & \cD_{10} & \cD_8 & \cD_9 & \cT_{1,3} & \cT_{1,1} & \cT_{1,2} & \cT_{2,3} & \cT_{2,1} & \cT_{2,2}  \\
19 & \cT_{2,1} & \cT_{2,2} & \cT_{2,3} & \cdots & \cD_5 & \cD_6 & \cD_7 & \cD_8 & \cD_9 & \cD_{10} & \cT_{1,1} & \cT_{1,2} & \cT_{1,3} \\
20 & \cT_{2,2} & \cT_{2,3} & \cT_{2,1} & \cdots & \cD_6 & \cD_7 & \cD_5 & \cD_9 & \cD_{10} & \cD_8 & \cT_{1,2} & \cT_{1,3} & \cT_{1,1}  \\
21 & \cT_{2,3} & \cT_{2,1} & \cT_{2,2} & \cdots & \cD_7 & \cD_5 & \cD_6 & \cD_{10} & \cD_8 & \cD_9 & \cT_{1,3} & \cT_{1,1} & \cT_{1,2}
\end{array}
$$

\end{example}

\noindent
{\bf Case 2:} $n \equiv 1~(\text{mod}~6)$.

In this case the direct products are not obtained in a straightforward way like
in Case 1 for $n \equiv 5~(\text{mod}~6)$. But, the general method is very similar.
We start with an example for $n=7$.

\begin{example}
\label{lem:DB7}
Continue with the example for $n=7$.
The last seven rows of the following array  form one of the possible direct products to
produce seven sets of quadruples from configuration $(2,2)$ for the DB Construction.
Taking the direct products of the column labels and the corresponding entries in the first six rows
contains all the quadruples obtained from the DLS Construction (and also some other quadruples).
$$
\begin{array}{ccccccccccccc}
\cR_{1,1} & \cR_{1,2} & \cR_{1,3} & \cR_{2,1} & \cR_{2,2} & \cR_{2,3} & \cR_{3,1} & \cR_{3,2} & \cR_{3,3} & \cC_3 & \cC_4 & \cC_5 & \cC_6 \\ \hline
\cT_{1,1} & \cT_{1,2} & \cT_{1,3} & \cT_{2,1} & \cT_{2,2} & \cT_{2,3} & \cT_{3,1} & \cT_{3,2} & \cT_{3,3} & \cD_3 & \cD_4 & \cD_5 & \cD_6 \\
\cT_{1,2} & \cT_{1,3} & \cT_{1,1} & \cT_{2,2} & \cT_{2,3} & \cT_{2,1} & \cT_{3,2} & \cT_{3,3} & \cT_{3,1} &  & \cD_3 & \cD_4 & \cD_5 \\
\cT_{1,3} & \cT_{1,1} & \cT_{1,2} & \cT_{2,3} & \cT_{2,1} & \cT_{2,2} & \cT_{3,3} & \cT_{3,1} & \cT_{3,2} &  &  & &  \\
\cD_6 &  &  & \cT_{1,1} & \cT_{1,2} & \cT_{1,3} & \cT_{2,1} & \cT_{2,2} & \cT_{2,3} & \cT_{3,1} &  & &  \\
 & \cD_6 &  & \cT_{1,2} & \cT_{1,3} & \cT_{1,1} & \cT_{2,2} & \cT_{2,3} & \cT_{2,1} & \cT_{3,2} &  & &  \\
 & & \cD_6 & \cT_{1,3} & \cT_{1,1} & \cT_{1,2} & \cT_{2,3} & \cT_{2,1} & \cT_{2,2} & \cT_{3,3} &  & &  \\ \hline
\cT_{3,1} & \cT_{3,2} & \cT_{3,3} & \cD_3 & \cD_4 & \cD_5 & \cD_6 & \cT_{1,2} & \cT_{1,3} & \cT_{2,3} & \cT_{2,2} & \cT_{2,1} & \cT_{1,1} \\
\cT_{3,2} & \cT_{3,3} & \cT_{3,1} & \cD_4 & \cD_5 & \cD_3 & \cT_{1,2} & \cD_6 & \cT_{1,1} & \cT_{2,2} & \cT_{1,3} & \cT_{2,3} & \cT_{2,1} \\
\cT_{3,3} & \cT_{3,1} & \cT_{3,2} & \cD_5 & \cD_3 & \cD_4 & \cT_{1,3} & \cT_{1,1} & \cD_6 & \cT_{2,1} & \cT_{2,3} & \cT_{1,2} & \cT_{2,2} \\
\cT_{2,1} & \cT_{2,2} & \cT_{2,3} & \cD_6 & \cT_{3,2} & \cT_{3,3} & \cD_3 & \cD_4 & \cD_5 & \cT_{1,1} & \cT_{1,2} & \cT_{1,3}& \cT_{3,1} \\
\cT_{2,2} & \cT_{2,3} & \cT_{2,1} & \cT_{3,2} & \cD_6 & \cT_{3,1} & \cD_4 & \cD_5 & \cD_3 & \cT_{1,2} & \cT_{1,1}& \cT_{3,3} & \cT_{1,3} \\
\cT_{2,3} & \cT_{2,1} & \cT_{2,2} & \cT_{3,3} & \cT_{3,1} & \cD_6 & \cD_5 & \cD_3 & \cD_4 & \cT_{1,3} & \cT_{3,2}& \cT_{1,1}& \cT_{1,2} \\
\cD_3 & \cD_4 & \cD_5 & \cT_{3,1} & \cT_{3,3} & \cT_{3,2} & \cT_{1,1} & \cT_{1,3} & \cT_{1,2} & \cD_6 & \cT_{2,1}& \cT_{2,2}& \cT_{2,3}
\end{array}
$$
\end{example}


If $n > 7$ consider a $(2n-1) \times (2n-1)$ array $\cH_{2n-1}$ ($\cH$ in short) whose first $\frac{3n-3}{2}$ columns are indexed by
$\cR_{i,j}$, $1 \leq i \leq \frac{n-1}{2}$, $1 \leq j \leq 3$,
and whose last $\frac{n+1}{2}$ columns are indexed by $\cC_i$, $\frac{n-1}{2} \leq i \leq n-1$.
The entries of the array are the elements $\cT_{i,j}$, $1 \leq i \leq \frac{n-1}{2}$, $1 \leq j \leq 3$,
and $\cD_i$, $\frac{n-1}{2} \leq i \leq n-1$. The first $\frac{3n-3}{2}$ columns of the first six rows of $\cH$ are defined
by abused of definition similar for $\cH_{2n-1}$,
in the case of $n\equiv 5$ (mod 6). That is, we choose the first two rows and the first $\frac{n-1}{2}$ columns
of the Latin square $\cM$ and substitute $\cR_i$ by $\cR_{i,1},\cR_{i,2},\cR_{i,3}$,
and replace every entry $\cT_\ell$ with a Latin square with elements $\cT_{\ell,1},\cT_{\ell,2},\cT_{l,3}$,
but note that the first three columns of the resultant rows
4, 5, and 6, form a $3 \times 3$ Latin square with the
elements $\cD_{n-3}$, $\cD_{n-2}$, $\cD_{n-1}$.
The other entries of the first six rows of $\cH$ are defined as follows.
\begin{itemize}
\item $\cH(1,\cC_i)= \cD_i$, $\frac{n-1}{2} \leq i \leq n-1$.

\item $\cH(2,\cC_i)= \cD_{i-1}$, $\frac{n+1}{2} \leq i \leq n-1$, and $\cH(2,\cC_{(n-1)/2}))=\cD_{n-1}$.

\item $\cH(3,\cC_i)= \cD_{i-2}$, $\frac{n+3}{2} \leq i \leq n-1$, $\cH(3,\cC_{(n-1)/2}))=\cD_{n-2}$, and $\cH(3,\cC_{(n+1)/2}))=\cD_{n-1}$.

\item Columns $\cC_{(n-1)/2}$, $\cC_{(n+1)/2}$, $\cC_{(n+3)/2}$, in rows 4, 5, and 6, form a $3 \times 3$ Latin square with the
elements $\cT_{(n-1)/2,j}$, $1 \leq j \leq 3$.

\item $\cH(4,\cC_i)= \cD_{i-3}$, $\frac{n+5}{2} \leq i \leq n-1$.

\item $\cH(5,\cC_i)= \cD_{i-4}$, $\frac{n+7}{2} \leq i \leq n-1$, and $\cH(5,\cC_{(n+5)/2}))=\cD_{n-4}$.

\item $\cH(6,\cC_i)= \cD_{i-5}$, $\frac{n+9}{2} \leq i \leq n-1$, $\cH(6,\cC_{(n+5)/2}))=\cD_{n-5}$, and $\cH(6,\cC_{(n+7)/2}))=\cD_{n-4}$.
\end{itemize}

One can easily verifies that these six rows form a $6 \times (2n-1)$ Latin rectangle.
It can be completed to the $(2n-1) \times (2n-1)$ Latin square $\cH$,
whose last $2n-7$ rows define the disjoint quadruples from configuration $(2,2)$ for the DB Construction.
Thus, we have proved Theorem~\ref{thm:main} for $n \equiv 1~(\text{mod}~6)$.

\begin{example}
\label{ex:LS19}
The six rows of the Latin rectangle for $n=19$ is given in the following array.
\begin{footnotesize}
$$
\begin{array}{cccccccccccccccccc}
\cR_{1,1} & \cR_{1,2} & \cR_{1,3} & \cR_{2,1} & \cR_{2,2} & \cR_{2,3} & \cdots & \cR_{9,1} & \cR_{9,2} & \cR_{9,3} & \cC_9 & \cC_{10} & \cC_{11} & \cC_{12} & \cC_{13} & \cdots & \cC_{17} & \cC_{18} \\ \hline
\cT_{1,1} & \cT_{1,2} & \cT_{1,3} & \cT_{2,1} & \cT_{2,2} & \cT_{2,3} & \cdots & \cT_{9,1} & \cT_{9,2} & \cT_{9,3} & \cD_9 & \cD_{10} & \cD_{11} & \cD_{12} & \cD_{13} & \cdots & \cD_{17} & \cD_{18} \\
\cT_{1,2} & \cT_{1,3} & \cT_{1,1} & \cT_{2,2} & \cT_{2,3} & \cT_{2,1} & \cdots & \cT_{9,2} & \cT_{9,3} & \cT_{9,1} & \cD_{18} & \cD_9 & \cD_{10} & \cD_{11} & \cD_{12} & \cdots & \cD_{16} & \cD_{17} \\
\cT_{1,3} & \cT_{1,1} & \cT_{1,2} & \cT_{2,3} & \cT_{2,1} & \cT_{2,2} & \cdots & \cT_{9,3} & \cT_{9,1} & \cT_{9,2} & \cD_{17} & \cD_{18} & \cD_9 & \cD_{10} & \cD_{11} & \cdots & \cD_{15} & \cD_{16} \\
\cD_{16} & \cD_{17} & \cD_{18} & \cT_{1,1} & \cT_{1,2} & \cT_{1,3} & \cdots & \cT_{8,1} & \cT_{8,2} & \cT_{8,3} & \cT_{9,1} & \cT_{9,2} & \cT_{9,3} & \cD_9 & \cD_{10} & \cdots & \cD_{14} & \cD_{15} \\
\cD_{17} & \cD_{18} & \cD_{16} & \cT_{1,2} & \cT_{1,3} & \cT_{1,1} & \cdots & \cT_{8,2} & \cT_{8,3} & \cT_{8,1} & \cT_{9,2} & \cT_{9,3} & \cT_{9,1} & \cD_{15} & \cD_9 & \cdots & \cD_{13} & \cD_{14} \\
\cD_{18} & \cD_{16} & \cD_{17} & \cT_{1,3} & \cT_{1,1} & \cT_{1,2} & \cdots & \cT_{8,3} & \cT_{8,1} & \cT_{8,2} & \cT_{9,3} & \cT_{9,1} & \cT_{9,2} & \cD_{14} & \cD_{15} & \cdots & \cD_{12} & \cD_{13} \\
\end{array}
$$
\end{footnotesize}
\end{example}

Finally, note that we have solved the case where $n \equiv 5~(\text{mod}~6)$ in exactly the same way as the
case $n \equiv 1~(\text{mod}~6)$. The advantage of the presented solution for $n \equiv 5~(\text{mod}~6)$
is that it provides an explicit construction of
the Latin square $\cH_{2n-1}$ from a Latin square $\cM$ of order $\frac{2n-1}{3}$.

\section{Conclusion, Discussion, and Future Research}
\label{sec:conclusion}

We have presented a recursive construction for PDQs
of order $4n$ from a set of PDQs
or order $2n$, where $n$ is congruent to 1 or 5 modulo 6. Having almost a large set
of PDQs order $2n$ the construction would yield an almost large set of PDQs order $4n$.
More precisely, if $n>5$ is an odd integer not divisible by 3, then
$$
D(4n) \geq 2n + \min \{ D(2n),2n-7 \}~.
$$

The first question that one might ask is whether the recursive formula can be improved to
$D(4n) = 2n +  \min \{ D(2n),2n-3 \}$ by partitioning the unused quadruples of configuration $(2,2)$
into four sets, each one of size $n^2(2n-1)$,
in which no triple is contained in more than one quadruple. Unfortunately, the answer
is negative. To see that one can easily verify that all the quadruples
of the form
$$
\{ (v,0),(x,0),(y,1),(z,1) \},~~\{v,x\} \subset \cA_1 \cup \cB_1 \cup \cA_2 \cup \cB_2, ~~\{y,z\} \subset \cD_0
$$
are not used in our construction,
which implies that $\cA_1 \cup \cB_1 \cup \cA_2 \cup \cB_2$ should be partitioned into four one-factors of $\Z_{2n}$
which is obviously impossible.

Next, we would like to know how the main construction is compared to previous results.
In~\cite{Etz93} we have the recursive formula
$$
D(4n) \geq 2n + \min \{ D(2n),n-2 \}~\text{for}~n \equiv 1~\text{or}~5~(\text{mod}~6)
$$
which is clearly inferior to the new result. An interesting comparison is with the lower bound on
$D(n)$ given in~\cite{EtHa91}. For $n \equiv 1$ or $5~(\text{mod}~6)$, it is proved that
$D(4n) \geq 3r$ if there exists a set of $r$ mutually 2-chromatic PDQs of order $2n$.
When $r=n$ we have that the set of $3n$ PDQs is maximal,
in the sense that the set cannot be extended to contain more PDQs. Our new result is better.
If there exists a set of $r$ PDQs of order $2n$, Theorem~\ref{thm:main} implies
that $D(4n) \geq 2n+r$ and if $r = n$ then the two results coincide, but while
the set of PDQs in~\cite{EtHa91} is maximal, our set might be extended. In fact,
if $D(2n) > n$, then Theorem~\ref{thm:main} implies $D(4n) > 3n$, while the
construction in~\cite{EtHa91} cannot be applied since in this case the PDQs of order
$2n$ cannot be mutually 2-chromatic.

There are many open questions which are remained to be answered in the future.
We list some of them in order of difficulty, from the most difficult one to the
easiest one to our opinion.
\begin{enumerate}
\item Construct an infinite family of large sets of PDQs of order $n$.
We conjecture that $n = 2^m$, $m \geq 4$, should be the easiest case, if there is an easy one.

\item Construct one nontrivial large set of SQSs of order $n$. We conjecture that $n=16$
should be the easiest target.

\item Prove that $D(2n) \geq n + D(n)$ for $n \equiv 2$ or $4~(\text{mod}~6)$.

\item Prove that $D(2n) \geq n + D(n)$ for an infinite sequence of values of $n$.

\item Prove that $D(2n) \geq n$ for all $n \equiv 1$ or $2~(\text{mod}~3)$ and $n > 4$.
Note that this holds for $n \equiv 2$ or $4~(\text{mod}~6)$.

\item Improve Theorem~\ref{thm:main} for all $n \equiv 2$ or $4~(\text{mod}~6)$.
\end{enumerate}



\end{document}